\documentclass[11pt]{article}

\usepackage[T1]{fontenc}
\usepackage{lmodern,amsmath,amssymb,amsthm,microtype}
\usepackage[margin=1in]{geometry}
\usepackage{parskip}
\usepackage[colorlinks=true,linkcolor=blue,citecolor=blue,urlcolor=blue]{hyperref}

\numberwithin{equation}{section}
\newtheorem{theorem}{Theorem}
\newtheorem{lemma}[theorem]{Lemma}

\newcommand{\R}{\mathcal R}
\newcommand{\E}{\mathbb E}
\newcommand{\Var}{\operatorname{Var}}
\newcommand{\dist}{\operatorname{dist}}
\newcommand{\Z}{\mathbb Z}

\title{Approximating 1 by Sums of Unit Fractions}
\author{Samuel Korsky}
\date{July 7, 2026}

\begin{document}
\maketitle

\begin{abstract}
\noindent If a finite collection (with repetitions allowed) of unit fractions has sum larger than $K$, how close can a sum of some of them come to $1$ from below? Improving a bound of Erd\H{o}s and Graham, we prove that the gap is at most $\exp(-c\sqrt{K\log K})$ for an absolute constant $c>0$ and all sufficiently large $K$. The proof combines equal fractions and applies a Fourier estimate to groups of denominators with comparable contributions. Elementary divisor counts then give the bound.
\end{abstract}

\section{Introduction}

Let $A$ be a finite multiset of positive integers. A subset of a multiset always respects its multiplicities. Write
\[
 \R(A)=\sum_{a\in A}\frac1a,
 \qquad
 \varepsilon(A)=1-\max\{\R(S):S\subseteq A,\ \R(S)\le1\}.
\]
Thus $\varepsilon(A)$ is the gap between $1$ and the closest reciprocal subsum from below. It is zero if some subset has reciprocal sum exactly $1$.

Erd\H{o}s and Graham proved that $\R(A)>K$ implies $\varepsilon(A)\ll K^{-2}$, and asked whether the gap is always at most $e^{-cK}$ for an absolute $c>0$ \cite[Section~4]{EG}. We prove the following bound.

\begin{theorem}\label{thm:main}
There are absolute constants $c>0$ and $K_0$ such that every finite multiset $A$ of positive integers with $\R(A)>K\ge K_0$ satisfies
\[
 \varepsilon(A)\le\exp\bigl(-c\sqrt{K\log K}\bigr).
\]
\end{theorem}

Related work concerns exact representations of $1$ using distinct denominators. Croot developed Fourier methods for unit fractions in short intervals and proved the finite-coloring conjecture of Erd\H{o}s and Graham \cite{Croot01,Croot03}. Bloom proved a density theorem for subsets with reciprocal sum $1$ \cite{Bloom}, and Liu and Sawhney obtained a quantitative density threshold \cite{LS}. Here we seek an approximation that applies to an arbitrary multiset with large reciprocal sum.

The proof begins by removing a closest sum to $1$ and combining equal fractions in what remains. If $N=1/\varepsilon(A)$, all the resulting denominators are smaller than $N$, and each occurs fewer times than its least prime factor. We then use one estimate for any group of denominators whose contributions $m_n/n$ are comparable. The estimate depends on how many denominators in the group can divide a single integer of size at most $N^3$. Elementary counts, applied separately to primes and composites, give
\[
 K-1<\sum_n\frac{m_n}{n}\ll\frac{x^2}{\log x},
 \qquad x=\log N.
\]
This implies the theorem. We first carry out this deduction, then prove the Fourier estimate in Section~\ref{sec:fourier}.

All logarithms are natural. We write $U\ll V$ if $U\le CV$ for an absolute constant $C$, and $U\asymp V$ if both $U\ll V$ and $V\ll U$. These constants may change from line to line.

\section{Combining equal fractions}\label{sec:reduction}

Assume $\varepsilon(A)>0$. Choose $S\subseteq A$ with $\R(S)=1-\varepsilon(A)$, and put
\[
 B=A\setminus S,\qquad N=\varepsilon(A)^{-1},\qquad x=\log N.
\]
Then $\R(B)>\R(A)-1$. Every denominator $n$ in $B$ is smaller than $N$: if $n\ge N$, adding $1/n$ to $\R(S)$ would give a larger sum still at most $1$, contrary to the choice of $S$.

Whenever a prime $p$ divides $n$ and at least $p$ copies of $1/n$ are present, replace $p$ of them by the single fraction
\[
 \underbrace{\frac1n+\cdots+\frac1n}_{p\ \mathrm{copies}}
 =\frac1{n/p}.
\]
It is clear that this process terminates.

Let $m_n$ be the resulting multiplicity of the fraction $1/n$, and let $P^-(n)$ denote the least prime factor of $n$. No denominator $1$ occurs, since it would represent a subset of $A$ with sum $1$. At termination, $m_n$ is smaller than every prime divisor of $n$. We have therefore obtained
\begin{equation}\label{eq:standing}
 2\le n<N,\qquad m_n<P^-(n),\qquad
 \sum_n\frac{m_n}{n}>K-1.
\end{equation}
Moreover, this multiset has no reciprocal subsum in
\begin{equation}\label{eq:gap}
 (1-N^{-1},1),
\end{equation}
since such a subsum would be closer to $1$ than $\R(S)$.

It remains to bound the sum in \eqref{eq:standing} by $O(x^2/\log x)$. We may assume that $x$ is sufficiently large: $m_n<n$ gives $\sum_n m_n/n<N=e^x$, so a fixed upper bound on $x$ would also bound $K$.

\section{An estimate for groups of denominators}\label{sec:group}

For a set $E$ of distinct denominators smaller than $N$, define
\begin{equation}\label{eq:D}
 D(E)=\max_{\substack{1\le q\le N^3\\q\in\Z}}
       \#\{n\in E:n\mid q\}.
\end{equation}
Thus $D(E)$ is the largest number of denominators in $E$ that divide one positive integer at most $N^3$. For nonempty $E$, we have $D(E)\ge1$. We will also use
\[
 D(E')\le D(E)\quad(E'\subseteq E),\qquad
 D(E_1\cup E_2)\le D(E_1)+D(E_2).
\]

\begin{lemma}\label{lem:group}
Suppose a multiset has denominators $2\le n<N$ and has no reciprocal subsum in $(1-N^{-1},1)$. Put $x=\log N$. For all sufficiently large $x$, if a set $E$ of its distinct denominators satisfies
\[
 \frac\alpha4<\frac{m_n}{n}\le\alpha\qquad(n\in E)
\]
for some $\alpha>0$, then
\begin{equation}\label{eq:group}
 \sum_{n\in E}\frac{m_n}{n}
 \ll1+D(E)\bigl(\sqrt{\alpha x}+\alpha x\bigr).
\end{equation}
\end{lemma}

The lemma does not require $m_n<P^-(n)$. That condition enters only when we apply the lemma to the multiset from Section~\ref{sec:reduction}.

The idea of the proof is as follows: If $E$ is too large, form a random subsum with expectation $1$ by choosing, independently for each denominator, either no copies or a uniformly distributed number of copies from a suitable interval. Near zero, its Fourier transform is close to that of a normal random variable. At larger frequencies, the transform is small unless many denominators divide integers close to the frequency divided by $2\pi$. The number $D(E)$ bounds how many denominators can do this. Fourier inversion then gives positive probability to the forbidden interval. The details are in Section~\ref{sec:fourier}.

\section{Bounding the total reciprocal sum}\label{sec:arithmetic}

We apply Lemma~\ref{lem:group} to the multiset satisfying \eqref{eq:standing} and \eqref{eq:gap}.

\subsection{Prime denominators}

Every positive integer $q\le N^3$ has $O(x/\log x)$ distinct prime factors. Indeed, if it has $r$ of them, their product is at least $(r+1)!$, so
\[
 \log((r+1)!)\le\log q\le3x,
\]
hence $r\ll x/\log x$.

Group the prime denominators according to
\[
 \frac\alpha2<\frac{m_p}{p}\le\alpha,
 \qquad \alpha=1,\frac12,\frac14,\ldots.
\]
There are $O(x)$ nonempty groups, since $1/N<m_p/p<1$, and each satisfies $D(E)\ll x/\log x$. Summing \eqref{eq:group} and the geometric series in $\sqrt\alpha$ and $\alpha$ gives
\begin{equation}\label{eq:primes}
 \sum_{p\ \mathrm{prime}}\frac{m_p}{p}
 \ll x+\frac{x}{\log x}(\sqrt x+x)
 \ll\frac{x^2}{\log x}.
\end{equation}

\subsection{Composite denominators}

For composite $n$, the inequality $m_n<P^-(n)\le\sqrt n$ shows that the denominators $n\le2x^2$ contribute at most
\[
 \sum_{2\le n\le2x^2}\frac1{\sqrt n}\ll x.
\]
Those occurring just once also contribute $O(x)$, since $\sum_{n<N}1/n\ll x$.

For the remaining composite denominators, group both $n$ and $m_n$ between successive powers of two. Let $E_{P,u}$ consist of those with
\[
 n>2x^2,\qquad P\le n<2P,\qquad u\le m_n<2u,
\]
where $P$ and $u$ are powers of two and $u\ge2$. A nonempty group satisfies
\begin{equation}\label{eq:rough}
 P>x^2,\qquad P^-(n)>u,\qquad u<\sqrt{2P}.
\end{equation}

\paragraph{Counting divisors.}
Write
\[
 Y=\log(2P),\qquad v=\log u,\qquad \beta=\frac{Cx}{Y},
\]
where $C$ is a sufficiently large absolute constant. Notice that $0<v<Y/2$ and $Y\le x+\log2$.

Fix a positive integer $q\le N^3$, and let $J$ be the number of its prime factors exceeding $u$, counted with multiplicity. Then $J\le3x/v$. Every divisor in $E_{P,u}$ uses at most $r=\lfloor Y/v\rfloor$ of these factors. Labeling repeated prime factors separately can only increase the count, so
\[
 \#\{n\in E_{P,u}:n\mid q\}
 \le\sum_{j=0}^r\binom Jj
 \le\bigl[e(1+J/r)\bigr]^r.
\]
For the last inequality, if $r<J$, set $z=r/J$ and use $z^r\sum_{j\le r}\binom Jj\le(1+z)^J\le e^r$; if $r\ge J$, use $2^J\le e^r$. Since $Y/v>2$, we have $r\ge Y/(2v)$ and $J/r\le6x/Y$. Choosing $C$ large enough therefore gives
\begin{equation}\label{eq:divisors}
 D(E_{P,u})\le\beta^{Y/v},\qquad \beta\ge e.
\end{equation}

\paragraph{Applying the lemma.}
Within $E_{P,u}$, put $\alpha=2u/P$. Then
\[
 \frac\alpha4<\frac{m_n}{n}\le\alpha,
 \qquad \alpha x<2\sqrt2,
\]
by \eqref{eq:rough}. Lemma~\ref{lem:group} gives
\[
 \sum_{n\in E_{P,u}}\frac{m_n}{n}\ll1+B(P,u),
 \qquad B(P,u)=\beta^{Y/v}\sqrt{\frac{xu}{P}}.
\]
There is also the elementary bound $\sum_{n\in E_{P,u}}m_n/n\le2u$, since $[P,2P)$ contains at most $P$ integers.

To use the smaller bound without adding a constant for every pair $(P,u)$, apply the lemma once for each value of $\alpha$: take the union of the groups with $2u/P=\alpha$ and $B(P,u)\le u$. The ratios $m_n/n$ in this union remain between $\alpha/4$ and $\alpha$, and the inequality for $D$ under unions gives a contribution of at most
\[
 \ll1+\sum_{\substack{2u/P=\alpha\\B(P,u)\le u}}B(P,u).
\]
For the other groups use $2u$. There are only $O(x)$ possible values of $\alpha$, since they are powers of two with $4/N<\alpha<4$. Thus the remaining composite contribution is
\begin{equation}\label{eq:totalcomp}
 \ll x+\sum_P\sum_u\min\{u,B(P,u)\}.
\end{equation}
The sums need include only nonempty groups.

\paragraph{Summing the estimates.}
We claim that, for each $P$,
\begin{equation}\label{eq:fixedP}
 \sum_u\min\{u,B(P,u)\}\ll\beta^2.
\end{equation}
The plan is to use the bound $u$ up to $u=e\beta^2$ and a decreasing geometric bound thereafter. Put
\[
 L=\log x,\qquad \lambda=\log\beta\ge1,\qquad
 \Phi(v)=\frac{\lambda Y}{v}+\frac{L+v-Y}{2}.
\]
Since $P=e^Y/2$, we have $B(P,u)=\sqrt2\,e^{\Phi(v)}$ and $u=e^v$. The values of $v$ are spaced by $\log2$ and lie in $(0,Y/2)$.

Set $b=2\lambda+1$. The identity $\beta=Cx/Y$ gives
\[
 L=\lambda+\log Y-\log C
 \le\lambda+\log b+Y/b\le b+Y/b,
\]
using $\log(Y/b)\le Y/b$ and $\log(2\lambda+1)\le\lambda+1$. Consequently,
\[
 \Phi(b)-b=\frac12\left(L-b-\frac Yb\right)\le0.
\]
For $b<v<Y/2$,
\begin{align*}
 \Phi(v)-\Phi(b)
 &=(v-b)\left(\frac12-\frac{\lambda Y}{bv}\right)\\
 &\le(v-b)\left(\frac12-\frac{2\lambda}{b}\right)
 \le-\frac{v-b}{6}.
\end{align*}
We may therefore use $u=e^v$ for $v\le b$, and $B(P,u)\le\sqrt2\,e^b e^{-(v-b)/6}$ for $v>b$. Each part is a geometric sum of size $O(e^b)=O(\beta^2)$, proving \eqref{eq:fixedP}. If $b\ge Y/2$, only the first part is needed.

As $P$ runs through powers of two with $x^2<P<N$, the values $Y=\log(2P)$ are spaced by $\log2$ and exceed $2\log x$. Hence
\[
 \sum_P\beta^2
 \ll x^2\sum_{j\ge2\log x/\log2}\frac1{j^2}
 \ll\frac{x^2}{\log x}.
\]
Together with \eqref{eq:totalcomp} and the discarded $O(x)$ terms, this proves the same bound for all composite denominators.

\subsection{Deduction of the theorem}

Combining the prime and composite estimates gives
\[
 K-1<\sum_n\frac{m_n}{n}\ll\frac{x^2}{\log x}.
\]
First $x\gg\sqrt K$, so $\log x\gg\log K$ for large $K$. Substituting back gives $x\gg\sqrt{K\log K}$. Since $\varepsilon(A)=e^{-x}$, this proves Theorem~\ref{thm:main}, subject only to Lemma~\ref{lem:group}.

\section{Proof of the group estimate}\label{sec:fourier}

Fix $E$ as in Lemma~\ref{lem:group}, and write $D=D(E)$. We may assume that $E$ is nonempty, so $D\ge1$. Choose a sufficiently large absolute constant $B_0$, and put
\[
 \ell=B_0\max\{x,\alpha^{-1}\}.
\]
The denominators $n<2\ell$ contribute at most $2\alpha\ell$. Let $E'$ contain the remaining denominators, and put $T=\#E'$. We will prove
\begin{equation}\label{eq:Tbound}
 T\ll\ell+D\sqrt{\ell x}.
\end{equation}
This suffices, since it implies
\[
 \sum_{n\in E}\frac{m_n}{n}
 \le2\alpha\ell+\alpha T
 \ll\alpha\ell+\alpha D\sqrt{\ell x}
 \ll1+D\bigl(\sqrt{\alpha x}+\alpha x\bigr).
\]

Suppose for a contradiction that
\begin{equation}\label{eq:Tlarge}
 T\ge C_0\bigl(\ell+D\sqrt{\ell x}\bigr)
\end{equation}
for a sufficiently large absolute constant $C_0$. In particular, $T\ge8\ell$ and $T\ge4D$. Since $E'$ is nonempty, $2\ell<N$.

\subsection{A random subsum and its variance}

For $n\in E'$, define
\[
 r_n=\lfloor n/\ell\rfloor,\qquad
 G=\sum_{n\in E'}\frac{r_n}{n},\qquad \theta=\frac2G.
\]
The inequalities $n\ge2\ell$ and $1/\ell\le\alpha/4$ give
\begin{equation}\label{eq:rtheta}
 \frac1{2\ell}\le\frac{r_n}{n}\le\frac1\ell,
 \qquad 1\le r_n\le m_n,
 \qquad 2\ell\le\theta T\le4\ell,
 \qquad 0<\theta\le\frac12.
\end{equation}
Thus all the copies used below are available.

For each $n\in E'$, let $\xi_n$ equal $1$ with probability $\theta$ and $0$ otherwise, and let $V_n$ be uniform on $\{0,1,\ldots,r_n\}$. Make all these choices independently, and set
\[
 Y_n=\frac{\xi_n V_n}{n},\qquad X=\sum_{n\in E'}Y_n.
\]
Every outcome of $X$ is an allowed reciprocal subsum, and
\[
 \E\left[X\right]=\frac\theta2\sum_{n\in E'}\frac{r_n}{n}=1.
\]
The variables $\xi_n$ let us use all $T$ denominators while keeping the mean equal to $1$.

Write $w_n=r_n/n$, and put $W_n=Y_n-\E\left[Y_n\right]$, $v_n=\E\left[W_n^2\right]$, and $\sigma^2=\sum_n v_n$. Since $\E\left[V_n/n\right]=w_n/2$, conditioning on $\xi_n$ gives
\[
 v_n=\theta\Var(V_n/n)+\theta(1-\theta)\frac{w_n^2}{4}.
\]
Using $\theta\le1/2$ for the lower bound and $v_n\le\E\left[Y_n^2\right]\le\theta w_n^2$ for the upper bound, we obtain
\[
 \frac{\theta w_n^2}{8}\le v_n\le\theta w_n^2.
\]
It follows from \eqref{eq:rtheta} that
\begin{equation}\label{eq:moments}
 \sigma^2\asymp\ell^{-1},\qquad
 |W_n|\le\ell^{-1},\qquad
 \E\left[|W_n|^3\right]\le\frac{v_n}{\ell}.
\end{equation}
All sums and products in the rest of this section are over $n\in E'$.

\subsection{Fourier estimates}

Let
\[
 \psi(t)=\E\left[e^{it(X-1)}\right]=\prod_n f_n(t),
 \qquad f_n(t)=\E\left[e^{itW_n}\right].
\]
We compare $\psi$ near zero with $e^{-\sigma^2t^2/2}$, the Fourier transform of a centered normal random variable of variance $\sigma^2$. At larger frequencies we use $D$ to bound $|\psi|$ directly.

\paragraph{Near zero.}
For a sufficiently small fixed $\eta>0$, we claim that
\begin{equation}\label{eq:low}
 \bigl|\psi(t)-e^{-\sigma^2t^2/2}\bigr|
 \ll\frac{|t|^3}{\ell^2}e^{-ct^2/\ell}
 \qquad(|t|\le\eta\ell).
\end{equation}
Taylor expansion and \eqref{eq:moments} give
\[
 f_n(t)=1-\frac{v_nt^2}{2}
          +O\!\left(\frac{|t|^3v_n}{\ell}\right).
\]
Since $v_n\le\ell^{-2}$, each factor differs from $1$ by $O(\eta^2)$. We may therefore take its logarithm near $1$. The error from $\log(1+z)=z+O(z^2)$ is absorbed into the displayed remainder, because
\[
 t^4v_n^2\le\eta\,\frac{|t|^3v_n}{\ell}.
\]
Summing the logarithms yields
\[
 \psi(t)=\exp\!\left(-\frac{\sigma^2t^2}{2}
                +O\!\left(\frac{|t|^3\sigma^2}{\ell}\right)\right).
\]
The remainder has magnitude at most $C\eta\sigma^2t^2$. Choose $\eta$ small enough and use $|e^z-1|\le|z|e^{|z|}$. Together with $\sigma^2\asymp\ell^{-1}$, this proves \eqref{eq:low}.

\paragraph{Larger frequencies.}
We first bound a single factor. For $r\ge1$, let
\[
 F_r(s)=\frac1{r+1}\sum_{j=0}^r e^{ijs},\qquad
 d=\dist(s,2\pi\Z)\in[0,\pi].
\]
We will use
\begin{equation}\label{eq:factor}
 |1-\theta+\theta F_r(s)|
 \le\exp\{-c\theta\min(1,r^2d^2)\}
 \qquad(0<\theta\le1/2).
\end{equation}
It is enough to show
\[
 1-\Re F_r(s)
 =\frac1{r+1}\sum_{j=1}^r(1-\cos(jd))
 \gg\min(1,r^2d^2).
\]
For $r=1$, use $1-\cos d\ge2d^2/\pi^2$. If $r\ge2$ and $rd\le2\pi$, sum over $j\le\lfloor r/2\rfloor$; then $jd\le\pi$ and the same inequality gives a lower bound of size $r^2d^2$. If $rd>2\pi$, the geometric-sum formula gives
\[
 |F_r(s)|\le\frac1{(r+1)\sin(d/2)}
 \le\frac\pi{(r+1)d}<\frac12.
\]
Finally, $|1-\theta+\theta z|^2\le1-\theta(1-\Re z)$ for $|z|\le1$ and $\theta\le1/2$, proving \eqref{eq:factor}.

Now fix $\eta\ell\le|t|\le N^3$, and put
\[
 a=\frac{|t|}{2\pi},\qquad \delta_n=\dist(a,n\Z).
\]
Choose a fixed $\kappa>0$ smaller than $\eta/(2\pi)$ and $1/4$. For $0\le R\le\kappa\ell$,
\begin{equation}\label{eq:nearby}
 \#\{n\in E':\delta_n\le R\}\le(2R+1)D.
\end{equation}
Indeed, each such denominator divides an integer in $[a-R,a+R]$. These integers are positive by the choice of $\kappa$, and smaller than $N^3$ since $a\le N^3/(2\pi)$ and $\ell<N/2$. There are at most $2R+1$ of them, and each has at most $D$ divisors from $E'$.

Take
\[
 R=\min\left\{\kappa\ell,\frac{T}{8D}\right\}.
\]
Since $T\ge4D$, \eqref{eq:nearby} shows that at least $T/2$ of the denominators have $\delta_n>R$. For each of these, apply \eqref{eq:factor} with $s=t/n$. Since
\[
 \dist(t/n,2\pi\Z)=\frac{2\pi\delta_n}{n},
 \qquad \frac{r_n}{n}\asymp\frac1\ell,
\]
the corresponding factor has modulus at most $\exp(-c\theta R^2/\ell^2)$. Centering does not change its modulus. Multiplying the factors and using $\theta T\asymp\ell$ gives
\begin{equation}\label{eq:high}
 |\psi(t)|
 \le\exp\!\left\{-c\min\left(\ell,\frac{T^2}{D^2\ell}\right)\right\}
 \le N^{-2}
 \qquad(\eta\ell\le|t|\le N^3).
\end{equation}
The last inequality follows from $\ell\ge B_0x$ and \eqref{eq:Tlarge}, with $B_0$ and $C_0$ sufficiently large.

\subsection{A subsum in the forbidden interval}

To detect whether $X$ lies in $(1-N^{-1},1)$, use
\[
 g(y)=\max\{0,1-|2y+1|\}.
\]
This function is positive on $(-1,0)$ and zero elsewhere, including the endpoints. Thus it suffices to prove $\E\left[g(N(X-1))\right]>0$.

Use the Fourier transform convention $\widehat g(u)=\int_{\mathbb R}g(y)e^{-iuy}\,dy$. Direct integration gives
\begin{equation}\label{eq:tent}
 \widehat g(u)=\frac12e^{iu/2}
       \left(\frac{\sin(u/4)}{u/4}\right)^2,
 \qquad |\widehat g(u)|\ll\frac1{1+u^2},
\end{equation}
with the value at zero understood by continuity. Since $\widehat g$ is integrable, Fourier inversion gives
\begin{equation}\label{eq:inversion}
 \E\left[g(N(X-1))\right]
 =\frac1{2\pi N}\int_{\mathbb R}\widehat g(t/N)\psi(t)\,dt.
\end{equation}
Absolute convergence justifies interchanging expectation and integration.

Let $Z$ be normal with mean $1$ and variance $\sigma^2$. Its density on $[1-N^{-1},1]$ is comparable to $1/\sigma\asymp\sqrt\ell$, since $\ell<N/2$. Hence
\begin{equation}\label{eq:normalmain}
 \E\left[g(N(Z-1))\right]\asymp\frac{\sqrt\ell}{N}.
\end{equation}
We compare \eqref{eq:inversion} with the same integral for $Z$.

For $|t|\le\eta\ell$, \eqref{eq:low} and the boundedness of $\widehat g$ give an error of at most
\[
 \ll\frac1N\int_{\mathbb R}
       \frac{|t|^3}{\ell^2}e^{-ct^2/\ell}\,dt
 \ll\frac1N.
\]
For $\eta\ell<|t|\le N^3$, the integral involving $\psi$ is $O(N^{-2})$ by \eqref{eq:high} and the integrability of $\widehat g$. For $|t|>N^3$, use $|\psi|\le1$ and \eqref{eq:tent} to obtain
\[
 \frac1N\int_{|t|>N^3}|\widehat g(t/N)|\,dt
 \ll\int_{|u|>N^2}\frac{du}{u^2}\ll N^{-2}.
\]
The normal integral outside $|t|\le\eta\ell$ is $O(N^{-1}\sqrt\ell\,e^{-c\ell})$. Altogether,
\[
 \E\left[g(N(X-1))\right]
 =\E\left[g(N(Z-1))\right]
  +O\!\left(\frac1N+\frac1{N^2}
             +\frac{\sqrt\ell}{N}e^{-c\ell}\right).
\]
By \eqref{eq:normalmain}, the main term has size $\sqrt\ell/N$. Since $\ell\ge B_0x$ and $x$ is sufficiently large, it exceeds the error. Some outcome of $X$ therefore lies in $(1-N^{-1},1)$, a contradiction. This proves \eqref{eq:Tbound}, and hence Lemma~\ref{lem:group} and Theorem~\ref{thm:main}.

\section*{Acknowledgments}

Inspired by Liu and Sawhney’s use of random reciprocal sums and Croot’s Fourier argument relating large Fourier factors to divisibility, the author formulated the multiset compression and the subsequent random sum construction used here. The author then supplied this construction to GPT-5.5 Pro, which carried out the Fourier analysis needed to complete the argument. The author reviewed, modified, and verified the arguments and assumes full responsibility for the final form and correctness of the paper.


\begin{thebibliography}{9}

\bibitem{Bloom}
T. F. Bloom, \emph{On a density conjecture about unit fractions}, Journal of the European Mathematical Society \textbf{27} (2025), no.~11, 4563--4589.

\bibitem{Croot01}
E. S. Croot III, \emph{On unit fractions with denominators in short intervals}, Acta Arithmetica \textbf{99} (2001), no.~2, 99--114.

\bibitem{Croot03}
E. S. Croot III, \emph{On a coloring conjecture about unit fractions}, Annals of Mathematics \textbf{157} (2003), no.~2, 545--556.

\bibitem{EG}
P. Erd\H{o}s and R. L. Graham, \emph{Old and New Problems and Results in Combinatorial Number Theory}, Monographies de L'Enseignement Math\'ematique, vol.~28, Universit\'e de Gen\`eve, 1980.

\bibitem{LS}
Y. P. Liu and M. Sawhney, \emph{On further questions regarding unit fractions}, International Mathematics Research Notices \textbf{2026}, no.~2, rnaf382.

\end{thebibliography}
\end{document}